\documentclass[epsfig,amsfonts]{amsart}
\usepackage{color}

\newcommand{\EE}{\ensuremath{\mathbb{E}}}

\newcommand{\NN}{\ensuremath{\mathbb{N}}}

\newcommand{\PP}{\ensuremath{\mathbb{P}}}

\newcommand{\RR}{\ensuremath{\mathbb{R}}}

\newcommand{\eps}{\ensuremath{\epsilon}}

\newcommand{\bB}{\ensuremath{\mathcal{B}}}

\newcommand{\dD}{\ensuremath{\mathcal{D}}}

\newcommand{\fF}{\ensuremath{\mathcal{F}}}

\newcommand{\tT}{\ensuremath{\mathcal{T}}}

   \newtheorem{lemma}{Lemma}[section]
   \newtheorem{theorem}[lemma]{Theorem}
   \newtheorem{remark}[lemma]{Remark}
   \newtheorem{example}[lemma]{Example}

   \newtheorem{definition}[lemma]{Definition}

\numberwithin{equation}{section}
\renewcommand{\phi}{\varphi}

\begin{document}

\title[Pathwise solutions and attractors for retarded SPDEs]
{Pathwise solutions and attractors for retarded SPDEs with time smooth diffusion coefficients}

\author{Hakima Bessaih}\address[Hakima Bessaih]{University of Wyoming
Laramie, WY 82071-3036, US} \email[Hakima Bessaih]{bessaih@uwyo.edu}

\author{Mar\'{\i}a J. Garrido-Atienza}\address[Mar\'{\i}a J. Garrido-Atienza]{Dpto. Ecuaciones Diferenciales y An\'alisis Num\'erico\\Universidad de Sevilla, Apdo. de Correos 1160, 41080-Sevilla, Spain} \email[Mar\'{\i}a J. Garrido-Atienza]{mgarrido@us.es}

\author{Bj{\"o}rn Schmalfu{\ss }}
\address[Bj{\"o}rn Schmalfu{\ss }]{Institut f\"{u}r Mathematik\\
Institut f{\"u}r Stochastik, Ernst Abbe Platz 2, 07737\\Jena,Germany\\ }\email[Bj{\"o}rn Schmalfu{\ss }]{bjoern.schmalfuss@uni-jena.de}

\begin{abstract}
In this paper we study the long--time dynamics  of mild solutions to retarded stochastic evolution systems driven by a Hilbert-valued Brownian motion.
As a preparation for this purpose we have to show the existence and uniqueness of a cocycle solution of such an equation.
We do not assume that the noise is given in additive form or that it is a very simple multiplicative noise. However, we need some smoothing property for the coefficient in front of the noise.
The main idea of this paper consists of expressing the stochastic integral in terms of non-stochastic integrals and the noisy path by using an integration by parts. This latter term causes that in a first moment only a local mild solution can be obtained, since in order to apply the Banach fixed point theorem it is crucial to have the H\"older norm of the noisy path to be sufficiently small. Later, by using appropriate stopping times, we shall derive the existence and uniqueness of a global mild solution. Furthermore, the asymptotic behavior is investigated by using the {\it Random Dynamical Systems theory}. In particular, we shall show that the global mild solution generates a random dynamical system that, under an appropriate smallness condition for the time lag, have associated a random attractor.
\end{abstract}

\maketitle

\section{Introduction}
The purpose of this paper is to show the existence of a random dynamical system generated by the solution of stochastic partial differential equations with delay of the following form
 \begin{equation}\label{eq}
  \left\{
  \begin{aligned}
    du&=(Au(t)+F(u_t))dt+G(u_t)dW(t),\qquad&&\text{for}\quad t\geq 0\\
    u(t)&=\xi(t),\qquad&&\text{for}\quad t\in[-\mu,0]
  \end{aligned}
  \right.
\end{equation}
in a separable Hilbert space $H$, where $A$ is the infinitesimal generator of an analytic semigroup on $H$, $F$ and $G$ are appropriate
nonlinear terms, and $W$ is a two-sided Wiener process with values in a separable Hilbert space $U$. The term $u_t$ is given by $u_t(s)=u(t+s)$ with $s\in [-\mu,0]$, where $\mu>0$ is given and the initial condition  is a continuous function on $[-\mu,0]$.\\
Retarded differential systems arise naturally in several situations in the area of applied mathematics due to biological
motivations like species growth or incubation time in delayed transmission of disease, see for instance \cite{Kuang} and \cite{Murray}, or due to
physical reasons with non--instant transmission phenomena such as high velocity fields in wind tunnel experiments, see \cite{HaleLunel}. Further examples can be found in biochemical reactions in the field of gene regulation where lengthy transcription has been modeled with delayed dynamics, see \cite{MBHT}. The asymptotic behavior of such models has meaningful interpretations like permanence, instability and chaotic developments.
From the mathematical point of view, there is a huge literature concerning the study of retarded stochastic differential systems, we refer here to the monographs by Mao \cite{Mao1,Mao2}, and to the papers \cite{MSch03}, \cite{real}, \cite{CaLiuTru}, \cite{CGR02} and \cite{GR03}, to mention a few of them.\\
In this paper we are interested in analyzing the long-time behavior of the (mild) solution to (\ref{eq}) by obtaining the random attractor associated to the random dynamical system generated by the mild solution. However, even when dealing with non-retarded equations, a fundamental problem in the study of the dynamics of a stochastic partial differential equation is to show that it generates a random dynamical system. Nevertheless, it is well-known that a large class of partial differential equations with stationary random coefficients and Ito stochastic ordinary differential equations generate random dynamical systems, see the monograph by Arnold \cite{Arn98}. However, for the stochastic partial differential equations  driven by Brownian motion the problem is much more difficult, and the reason is twofold: on the one hand, the stochastic integral is only defined almost surely where the exceptional set may depend on the initial state, which contradicts the definition of the cocycle property, and on the other, Kolmogorov's theorem in an appropriate form is only true for finite dimensional random fields, see Kunita \cite{Kun91} Theorem
1.4.1. In spite of that there are some partial results for additive as well as simple multiplicative Brownian noises, see for instance the papers \cite{FL}, \cite{DLSch, DLSch1} and \cite{CKSch}, to mention only a few of them. In the case of having retarded stochastic systems there are also positive results, as shown in the papers \cite{CGSV10}, \cite{CGSch07} and \cite{CRCh}. In all the mentioned articles the main ingredient consists of transforming the stochastic equation into a random one, being possible to deal with the latter by using deterministic techniques.  This transformation is known as {\it cohomology}, which consists of a stationary coordinate change by means of which flows of stochastic differential equations may be viewed as ordinary differential equations with a random parameter. This method presents the restriction that it is not always possible to find this appropriate change of variable, since it is applicable only when considering an additive noise or very particular cases of multiplicative noise.\\
Our technique is not based on the comohology, rather on considering diffusion terms $G$ with the smoothness property that the corresponding stochastic integral can be expressed, by means of the integration by parts formula, in terms of two {\it non-stochastic} integrals and the noise path as well (see formula (\ref{eq2}) below), which in particular also means that our delayed system can be reduced to a deterministic delayed system with a random parameter. This idea of removing the stochastic It\^o integral is partially borrowed from Bensoussan and Frehse \cite{BF}. Nevertheless, the main novelty in this article is the fact that we are able to consider non-trivial delayed diffusion terms, which is not at all a trivial problem as stressed by Mohammed \cite{Mo}. As pointed out before, we shall investigate the existence of mild solutions to (\ref{eq}) generating a random dynamical system. Due to the mentioned transformation of the stochastic integral, the mild solution will be given in terms of the noise path. This fact
allows the establishment of the existence and uniqueness of a {\it local} mild solution.
As it will be shown in Section 2, see condition (\ref{cond}), the Banach fixed point argument will ensure the existence and uniqueness of a mild solution provided that an appropriate H\"older--norm of the noise path is sufficiently small. As a consequence, we shall consider stopping times $\{ T_i \}_{i\in \NN}$ with the property that in every interval $[T_i,T_{i+1}]$ we can find a unique local mild solution, and thus, we finally can build a global mild solution for the problem ({\ref{eq}) by glueing all these local solutions.\\
Note that the idea of considering a smoothing diffusion term was also pointed out by Mohammed and Scheutzow \cite{MSch03}. In that paper they construct the infinite-dimensional stochastic semiflow generated by the solution of stochastic functional differential equations, but when having a $p$-dimensional Brownian motion and not a Hilbert-valued Brownian motion. There are more differences with respect to our paper, since they consider a different phase space than in our setting and do not cover the existence of the global attractor associated to the flow.\\
Once the existence of a mild solution is already established, we want to investigate its longtime behavior by analyzing the existence of random attractors associated to the random dynamical systems generated by the solution of (\ref{eq}). For an overview about the theory of random attractors we refer to \cite{Arn98}, \cite{CGSchV08}, \cite{CheSch}, \cite{FlanSch}, \cite{ImSch}, \cite{Sch00}, amongst many others. In particular we will obtain that, under an appropriate smallness condition for the time lag, there exists a tempered absorbing ball which will ensure the existence of a random attractor for our retarded system. \\
The content of the paper is as follows. In Section 2 we first establish the framework in which our analysis is carried out, introducing the basic notations and assumptions, and defining the mild solution as a sum of different terms in which there are no stochastic integrals. We also prove the existence and uniqueness of local solutions in adequate time intervals, that with the help of stoping times, will be sufficient to establish the existence and uniqueness of a global mild solution. This global solution generates a random dynamical system in the space $C([-\mu,0];H)$. We also exhibit an example to illustrate the different regularity conditions for the non-linear terms appearing in (\ref{eq}). Section 3 is devoted to the study of the random attractor associated to the random dynamical system obtained in the previous section.\\
\section{Pathwise solutions}

We start this section by introducing the abstract definition of a random dynamical system.

\begin{definition} \label{sem} Let $V$ be a Banach space. A mapping $\phi:\RR^+\times V\to V$
having the {\em semigroup} property
\begin{equation*}
    \phi(t,\cdot)\circ\phi(\tau,u_0)=\phi(t+\tau,u_0),\qquad\phi(0,u_0)=u_0\qquad\text{for }t,\,\tau\in\RR^+\quad\text{and }u_0\in V
\end{equation*}
is called an autonomous dynamical system. \end{definition}

We want to consider a generalization of the concept of an autonomous dynamical system to {\em non-autonomous} and {\em random dynamical systems}. As first
we introduce as a model for a noise a {\em metric dynamical system} $(\Omega,\fF,\PP,\theta)$ where  $(\Omega,\fF,\PP)$ is a probability space and $\theta$ is a $\bB(\RR)\otimes \fF,\fF$ measurable {\em flow} $\theta=(\theta_t)_{t\in\RR}$, i.e.
\begin{align*}
\theta_t\circ\theta_\tau=\theta_{t+\tau},\quad\theta_0={\rm id}_\Omega\quad\text{for }t,\tau\in \RR,\;\omega\in\Omega
\end{align*}
such that $\PP$ is ergodic with respect to $\theta$.

In the following we consider the {\em Brownian motion} metric dynamical system: let $U$ be a separable Hilbert space
and let $C_0(\RR;U)$ be the set of continuous functions on $\RR$ with values in $U$ which are zero at zero equipped with the compact open topology.
We consider the Wiener measure $\PP$ on $\bB(C_0(\RR;U))$ having a trace--class covariance operator $Q$ on $U$. Then  Kolmogorov's fundamental theorem and Kolmogorov's theorem about a (H{\"o}lder-)continuous version give the canonical probability space $(C_0(\RR;U),\bB(C_0(\RR;U)),\PP)$, which becomes an ergodic metric dynamical system if we add the Wiener shift
\begin{equation}\label{shift}
  \theta_t\omega(\cdot)=\omega(\cdot+t)-\omega(t),\qquad \omega\in \Omega.
\end{equation}
Let us consider for some fixed $\beta\in (0,1/2)$ the set of paths $\Omega$ in $C_0(\RR;U)$ which have a finite $\beta$-H{\"o}lder-seminorm on any interval $[-k,k],\,k\in\NN$. Denote by $\|\cdot\|_{\beta,a,b}$ (and very often simply by $\|\cdot\|_\beta$) the $\beta$-H\"older-seminorm on an interval $[a,b]$.
Again by Komogorov's theorem about a H{\"o}lder-continuous version,  this set contained in $\bB(C_0(\RR;U))$ has measure one, and in addition it is invariant with respect to $\theta=(\theta_t)_{t\in\RR}$ defined by (\ref{shift}). We choose $\fF$ the trace-$\sigma$-algebra of $\bB(C_0(\RR;U))$ with respect to $\Omega$, and for the restriction of $\PP$ to this new $\sigma$-algebra we use again the symbol $\PP$. In the following we will work with this metric dynamical system $(\Omega,\fF,\PP,\theta)$. We also note that from the above canonical Brownian motion it is not hard to derive  a filtered Brownian motion $(\Omega,\fF,(\fF_t)_{t\ge 0},\PP)$ or its corresponding $\PP$-completion $(\Omega,\bar\fF,(\bar\fF_t)_{t\ge 0},\bar\PP)$, where the filtration $(\bar\fF_t)_{t\ge 0}$ satisfies the {\em usual conditions}. \\

As a generalization of the semigroup property introduced in Definition \ref{sem}, we consider a {\em random dynamical system}, RDS for shorten, for some metric dynamical system $(\Omega,\fF,\PP,\theta)$, which is given by a $\bB(\RR^+)\otimes\fF\otimes \bB(V),\bB(V)$-measurable mapping
\begin{equation*}
    \phi:\RR^+\times\Omega\times V\to V
\end{equation*}
such that
\begin{align*}
\phi(0,\omega,u_0)&=u_0,\\
   \phi(t+\tau,\omega,u_0)&=\phi(t,\theta_\tau\omega,\cdot)\circ\phi(\tau,\omega,u_0),
    \text{ for all }t,\,\tau\in\RR^+,\;u_0\in V,\;\omega\in\Omega.
\end{align*}
We emphasize that  $t$ or $u_0$ dependent exceptional sets of $\PP$-measure zero, what are typical for the classical theory of stochastic differential equations, are not allowed in the definition of a random dynamical system.

The first problem that we will face below is to show that the retarded evolution system \eqref{eq} forms a random dynamical system.\\

 Next we introduce with details the retarded stochastic system we are interested in. Let $H$ be a separable Hilbert  space with norm $|\cdot|$ and, for some fixed $\mu>0$, let $C_\mu=C([-\mu,0];H)$ be the usual space of continuous functions. We consider the delayed stochastic partial differential equation \eqref{eq}
 interpreted in a mild sense: we look for a {\em mild} solution to \eqref{eq}, which means that we aim at solving in $C([-\mu,T];H)$ the following operator equation
\begin{equation}\label{eq1bis}
  u(t)=
  \begin{cases}
  S(t)\xi(0)+\displaystyle{\int_0^tS(t-r)F(u_r)dr+\int_0^tS(t-r)G(u_r)dW(r)}, \, t\in [0,T],\\
  \xi(t), \, t\in [-\mu,0],
  \end{cases}
\end{equation}
for the initial data $\xi \in C_\mu$ and for the $U$-valued Brownian motion $W$, defined over $(\Omega,\bar\fF,(\bar\fF_t)_{t\ge 0},\bar\PP)$ and with covariance given by a trace-class operator $Q$.

Now we describe the assumptions on the coefficients of this equation.
$-A$ is a strictly positive and symmetric operator
with a compact inverse generating a $C_0$ analytic--semigroup $S=(S(t))_{t\in\RR^+}$ on $H$. For $\gamma\geq 0$ we consider the spaces $D((-A)^\gamma)$, defined in the usual way, see Sell and You \cite{SellYou}, Chapter II, which are compactly embedded in $H$. Note that under these conditions the following estimates are satisfied: there exists a $\lambda>0$ such that
\begin{align}\label{eq5}
\begin{split}
|S(t)| &\leq M e^{-\lambda t},\;M\geq 1,\\
  |(-A)^\gamma S(t)|_{L(H)}& \le c_\gamma \frac{1}{t^\gamma}e^{-\lambda t},\;\gamma\in [0,1),\\
 |S(t)-{\rm id}|_{L(D((-A)^{\sigma}),D((-A)^{\theta}))} &\le c_{\sigma,\theta}
t^{\sigma-\theta}, \;\text{for }\theta\geq 0,\quad \sigma\in
[\theta,1+\theta]
\end{split}
\end{align}
for $t>0$. For simplicity, the previous constants $M$, $c_\gamma$ and $c_{\sigma,\theta}$ will be assumed to be equal to $1$. From these inequalities, for $0\leq q\leq r\leq s\leq t$, we can derive that
\begin{align}\label{eq30}
\begin{split}
|S(&t-r)-S(t-q)|_{L(D((-A)^{\delta}),D((-A)^{\gamma}))}\le c_{\delta,\gamma}(r-q)^\alpha(t-r)^{-\alpha-\gamma+\delta},
\end{split}
\end{align}
for $\gamma \geq 0$, $\alpha \in [0,1]$ and $\delta \geq \alpha$. The constant $c_{\delta,\gamma}$ is also assumed to be equal to $1$.\\

Now we describe the assumptions regarding the nonlinear terms $F$ and $G$. For $F:C_\mu\mapsto H$ we assume global Lipschitz continuity, that is, there exists $L_F>0$ such that
\begin{equation}\label{eq3}
|F(x)-F(y)|\le L_{F}\|x-y\|_\mu,\qquad\text{for } x,\,y\in C_\mu.
\end{equation}

For $G: C_\mu\mapsto L_{2,Q}(U; H)$ we also assume Lipschitz continuity in the corresponding spaces, i.e.,
\begin{equation}\label{G1}
\|G(x)-G(y)\|_{L_{2,Q}(U;H)}\leq L_{G}\|x-y\|_{\mu}, \qquad\text{for } x,\,y\in C_\mu,
\end{equation}
where $L_{2,Q}(U;H)$ denotes the Hilbert-Schmidt space of linear operators from $U$ to $H$ related to the trace-class operator $Q$, see Da Prato and Zabczyk \cite{DaPrato} Chapter 4. \\

For any $T>0$ let us consider the space $C([-\mu,T];H)$, and denote by  $|||u|||$ the norm on this space, avoiding the typing of $T$ in the previous norm for the sake of exposition. As usual, for $u\in C([-\mu,T];H)$ and $t\in [0,T]$, the term $u_t \in C_\mu$ is given by $u_t(s)=u(t+s)$ with $s$ belonging to the interval $[-\mu,0]$. This means that $u_t$ tracks the history of the process over the delay period.

These assumptions allow us to conclude the existence and uniqueness of a mild
solution of \eqref{eq1bis} in the sense of Da Prato and Zabczyk \cite{DaPrato} Chapter 7.
We refer to the paper of Taniguchi {\it et al.} \cite{taniguchi}, where the equation contains a delay similarly to ours, and \cite{KL}.

\begin{theorem} Assume that the operators  $F$ and $G$ satisfy the assumptions \eqref{eq3} and
\eqref{G1}. Then, there exists a unique global  stochastic process $(u(t))_{t\in[-\mu,T]}$ with paths in $C([-\mu,T];H)$ and $u(s)=\xi(s)$ for $s\in [-\mu,0]$ such that
the stochastic process $(u_t)_{t\in [0,T]}\in C_\mu$ is $(\bar\fF_t)_{t\in [0,T]}$-predictable and 
\eqref{eq1bis} holds for any $t\in[0,T]$ and for any $\bar \fF_0$-measurable random variable $\xi$ in $C_\mu$ almost surely.
\end{theorem}

Our purpose goes beyond this existence result. Precisely speaking, we want to show the existence of a cocycle version of the above solution. To do this we have to impose stronger conditions on $G$, in order to obtain a mild solution where  exceptional sets do not appear, since these sets contradict the definition of a cocycle.

We also assume that $G$ is smoothing in the following sense: for any $u\in C([-\mu,T];H)$ the mapping
\begin{equation*}
  [0,T]\ni t\mapsto G(u_t)\in C^1([0,T];L(U;H))
\end{equation*}
such that the derivative of this mapping is given by another operator $K$ with a special structure, that is
\begin{equation}\label{GK}
  \frac{d}{dt}G(u_t)=K(u_t)
\end{equation}
with $K:C_\mu \mapsto L(U;H)$ being Lipschitz continuous with Lipschitz constant denoted by $L_K$. In addition to the above conditions, let us also impose the stronger condition that  $G$ is Lipschitz continuous with values in $L(U;D((-A)^\nu))$, the space of bounded linear operators from $U$ to $D((-A)^\nu)$, where $\nu\in (0,1)$:
\begin{equation}\label{G3}
\|G(x)-G(y)\|_{L(U;D((-A)^\nu))}\leq L_{G,\nu}\|x-y\|_{\mu}, \qquad\text{for } x,\,y\in C_\mu.
\end{equation}
Note that from the above conditions we can just conclude that $G$ is a Lipschitz mapping from $C_\mu$ into $L(U;H)$.

For an example of non-linear terms $F$ and $G$ satisfying the previous assumptions see the Example \ref{example} below.

In the following we shall be able to get rid of the stochastic integral by applying an integration by parts formula, which allows us to handle our equation in a pathwise way, and which will turn out to be essential when proving the cocycle property for the solution operator to the stochastic delayed system (\ref{eq1bis}). 

In what follows $\omega(t)$, $t\in \RR$, represents the canonical version of the Brownian motion $W$.

Avoiding the stochastic integral is possible thanks to the existence of the above operator $K$, since then we can give the following interpretation to the stochastic integral
\begin{align}\label{eq2}
\begin{split}
  \int_0^tS(t-r)G(u_r)d\omega(r)&=G(u_t)\omega(t)+\int_0^tS(t-r)AG(u_r)\omega(r)dr
  \\
  &-\int_0^tS(t-r)K(u_r)\omega(r)dr,
  \end{split}
\end{align}
which follows by an  application of the integration by parts formula (note that $\omega(0)=0$ and that $S$ and the operator $A$ commute).
For more motivations to the integration by parts formula we refer to Bensoussan and Frehse \cite{BF}.

Now we are going to establish the existence and uniqueness of a mild solution to our delayed equation (\ref{eq}) by taking into account the expression of the stochastic integral given by (\ref{eq2}). As we shall prove, in a first step we obtain a {\it local} mild solution, in the sense that there exists a random variable in $\RR^+$, denoted by $T(\omega)$, such that there exists exactly one $u\in C([-\mu,T(\omega)];H)$ such that \eqref{eq1bis} is satisfied.

\begin{theorem}\label{t1}
Assume that $F, G$ satisfy the assumptions \eqref{eq3} and \eqref{G3} such that $G$ has the special structure \eqref{GK} and $K$ is Lipschitz continuous. Then, for any $\xi\in C_\mu$ and any $\omega\in\Omega$ there exists a mild local solution to \eqref{eq1bis}, that is, there exists $T(\omega)>0$ such that this equation has a unique mild solution $u\in C([-\mu,T(\omega)];H)$, i.e., $u$ satisfies
\begin{equation}\label{eq?}
u(t)=
\begin{cases}
S(t)\xi(0)+\displaystyle {\int_0^t S(t-\tau)F(u_\tau){d}\tau+\int_0^tS(t-r)AG(u_r)\omega(r)dr}\\
\displaystyle {+G(u_t)\omega(t)-\int_0^tS(t-r)K(u_r)\omega(r)dr},\quad \text{for}\quad t\in[0,T(\omega)],\\
    \xi(t),\quad \text{for}\quad t\in[-\mu,0],
\end{cases}
\end{equation}
and depends continuously on $\xi$.
\end{theorem}

\begin{proof}
We consider the complete metric subspace $C^\xi([-\mu,T];H)$ of functions  $u\in C([-\mu,T];H)$ with $u(s)=\xi(s)$ for $s\in [-\mu,0]$.
Let us write for a while $T$ instead of $T(\omega)$.
For such a $T>0$ to be determined later, consider the operator
$\mathcal{T}\colon C^\xi([-\mu,T];H)\rightarrow C^\xi([-\mu,T];H)$ defined, for $t\geq 0$, by
  \begin{align*}
   \mathcal{T} (u)(t)&=S(t)\xi(0)+\int_0^t S(t-\tau)F(u_\tau){d}\tau+G(u_{t})\omega(t)+\int_0^{t}S(t-r)AG(u_r)\omega(r)dr\\
   &-\int_0^{t}S(t -r)K(u_r)\omega(r)dr.
  \end{align*}

We want to check that this operator is a self-mapping and in addition a contraction, and thus it has a unique fixed point in $C^\xi([-\mu,T];H)$, where $T$ will be determined according to the Banach fixed point theorem. First, taking into account that $\omega \in \Omega$, the Lipschitz continuity of $F$, $G$ and $K$, and the property (\ref{eq5}), we obtain the above integrals define continuous mappings from $[0,T]$ into $H$ which are zero for the time parameter zero, see \cite{BF}. In particular, for the second integral we note that $[0,t]\ni r\mapsto |AS(t-r)G(u_r)\omega(r)|$ is integrable by the regularity of $G$. Therefore, it follows easily that $t\mapsto\tT(u)(t)$ is continuous such that $\tT$ maps $C^\xi([-\mu,T];H)$
into itself.

Now we take $u^1,\,u^2 \in C^\xi([-\mu,T];H)$. Due to the Lipschitz regularity of $F$, $G$ and $K$, for $t\in [0,T]$ we get
\begin{align*}
|(\mathcal T(u^1)- & \mathcal T(u^2))(t)| \leq \bigg(\bigg|\int_0^t S(t-\tau)(F(u^1_\tau)-F(u^2_\tau)){d}\tau\bigg|+|(G(u^1_{t})-G(u^2_{t}))\omega(t)|\\
& \quad +\bigg|\int_0^{t}(-A)^{1-\nu}S(t-r)(-A)^\nu(G(u^1_r)-G(u^2_r))\omega(r)dr\bigg|\\
& \quad +\bigg|\int_0^{t}S(t -r)(K(u^1_r)-K(u_r^2))\omega(r)dr\bigg|\bigg)\\
&\leq L_F\sup_{r\in[0,t]}\|u_r^1-u_r^2\|_\mu t+L_G \|u_t^1-u_t^2\|_\mu\|\omega\|_\beta t^\beta\\
& \quad +  \|\omega\|_\beta \int_0^t\frac{L_{G,\nu}\|u_r^1-u_r^2\|_\mu}{(t-r)^{1-\nu}}r^\beta dr+  \|\omega\|_\beta \int_0^t\ L_K \|u_r^1-u_r^2\|_\mu r^\beta dr\\
&\leq L_F\sup_{r\in[0,t]}\|u_r^1-u_r^2\|_\mu t+L_G \|u_t^1-u_t^2\|_\mu\|\omega\|_\beta t^\beta\\
&\quad+\sup_{r\in[0,t]}\|u_r^1-u_r^2\|_\mu (L_{G,\nu} t^{\beta+\nu} +L_K t^{\beta+1}) \|\omega\|_\beta,
\end{align*}
and therefore
\begin{align*}
|||\mathcal T(u^1)-\mathcal T(u^2)|||&\leq C|||u^1-u^2||| T+C|||u^1-u^2|||  (T^\beta+T^{1+\beta}+T^{\beta+\nu}) \|\omega\|_\beta,
\end{align*}

with $C$ a positive constant. Taking  $T:=T(\omega)$ small enough such that
\begin{align}\label{cond}
CT+C(T^\beta+T^{1+\beta}+T^{\beta+\nu}) \|\omega\|_\beta\leq 1/2,
\end{align}
 we have $|||\mathcal T(u^1)-\mathcal T(u^2)||| \leq 1/2 |||u^1-u^2|||$. Then the Banach fixed point theorem gives a solution to (\ref{eq1bis}).\\
Moreover, since the contraction constant is independent of $\xi$ we have that the solution depends continuously on $\xi$. Let us show this statement with some details. In order to prove the continuous dependence of the solution in the delay input $\xi$, consider also $\tilde \xi \in C_\mu$ and let $\tilde u\in C([-\mu,T];H)$ be the unique  solution to (\ref{eq1bis}) with initial condition $\tilde \xi$. Then, in a similar manner as we have proved the contraction property, for $t\geq 0$ we obtain
\begin{align*}
|u(t)-\tilde u(t)| &\leq \|\xi-\tilde \xi\|_\mu+C|||u-\tilde u|||  t+ C|||u-\tilde u|||  (t^\beta+t^{\beta+\nu}+t^{1+\beta}) \|\omega\|_\beta,
\end{align*}
therefore, taking $T$ small enough such that (\ref{cond}) holds, we obtain
\begin{align*}
|||u-\tilde u||| & \leq 2\|\xi-\tilde \xi\|_\mu +\frac{1}{2}|||u-\tilde u|||
\end{align*}
and thus the continuous dependence on the initial condition $\xi$ follows.

\end{proof}

Up to now we have been able to prove the existence of a unique {\it local} mild solution $u\in C([-\mu,T(\omega)];H)$ to (\ref{eq?}), since, as we have seen, for instance the term $G(u_t)\omega(t)$ produces a sort of Lipschitz constant depending on $\|\omega\|_\beta$. In what follows we will derive the existence of a {\it global} mild solution.

Denote $T_1(\omega)=T(\omega)$ and let us build the solution in the next time interval, say $[T_1(\omega), T_2(\omega)]$, i.e., we need to find $T_2(\omega)$ such that we also have a local mild solution in the last interval.  For $t\geq T_1(\omega)$ (which in the following computations will be denoted by $T_1$ for short), similarly to (\ref{eq2}) we get

\begin{align*}
\begin{split}
& \int_{T_1}^tS(t-r)G(u_r)d\omega(r)=G(u_t)\omega(t)-S(t-T_1)G(u_{T_1})\omega(T_1)\\
& \quad +\int_{T_1}^tS(t-r)AG(u_r)\omega(r)dr -\int_{T_1}^tS(t-r)K(u_r)\omega(r)dr\\
&=G(u_t)\omega(t)-S(t-T_1)G(u_{T_1})\omega(T_1)+\int_{0}^{t-T_1}S(t-T_1-r)AG(u_{T_1+r})\omega(r+T_1)dr\\
& \quad -\int_{0}^{t-T_1}S(t-T_1-r)K(u_{T_1+r})\omega(r+T_1)dr\\
&=G(u_t)\omega(t)-S(t-T_1)G(u_{T_1})\omega(T_1)+\int_0^{t-T_1} AS(t-T_1-r)G(u_{T_1+r})\theta_{T_1} \omega(r)dr\\
&\quad +\int_0^{t-T_1} AS(t-T_1-r)G(u_{T_1+r})\omega(T_1)dr\\
&\quad -\int_0^{t-T_1} S(t-T_1-r)K(u_{T_1+r})\theta_{T_1} \omega(r)dr-\int_0^{t-T_1} S(t-T_1-r)K(u_{T_1+r})\omega(T_1)dr\\
&=G(u_{T_1+(t-T_1)})\theta_{T_1}\omega(t-T_1)-\int_0^{t-T_1} S(t-T_1-r)K(u_{T_1+r})\theta_{T_1} \omega(r)dr\\
&\quad +\int_0^{t-T_1} AS(t-T_1-r)G(u_{T_1+r})\theta_{T_1} \omega(r)dr,
\end{split}
\end{align*}
where the last equality follows from the fact that
\begin{align}\label{tr}
\begin{split}
&-\int_0^{t-T_1} S(t-T_1-r)K(u_{T_1+r})\omega(T_1)dr\\
&+\int_0^{t-T_1} AS(t-T_1-r)G(u_{T_1+r})\omega(T_1)dr\\
=&-\int_0^{t-T_1} \frac{d}{dr}(S(t-T_1-r)G(u_{T_1+r}))\omega(T_1)dr\\
= &-G(u_t)\omega(T_1)+S(t-T_1)G(u_{T_1})\omega(T_1).
\end{split}
\end{align}

Therefore, we are interested in solving

\begin{equation*}
u(t)=
\begin{cases}
S(t-T_1(\omega))u(T_1(\omega))+\displaystyle{\int_0^{t-T_1} S(t-T_1-r)F(u_{r+T_1})dr}\\
\displaystyle{+G(u_{T_1+(t-T_1)})\theta_{T_1}\omega(t-T_1)-\int_0^{t-T_1} S(t-T_1-r)K(u_{T_1+r})\theta_{T_1} \omega(r)dr}\\
\displaystyle{+\int_0^{t-T_1} AS(t-T_1-r)G(u_{T_1+r})\theta_{T_1} \omega(r)dr}, \quad t-T_1(\omega)\geq 0,\\
    u_1(t),\quad t-T_1(\omega)\in[-\mu,0],
\end{cases}
\end{equation*}
where $u_1$ denotes the solution obtained on $[-\mu,T_1(\omega)]$. But solving the above system is equivalent to solve the problem for $y(s)=u(s+T_1(\omega))$
\begin{equation*}
y(s)=
\begin{cases}
S(s)\hat\xi(0)+\displaystyle{\int_0^{s} S(s-r)F(y_{r})dr+G(y_{s})\theta_{T_1}\omega(s)}\\
-\displaystyle{\int_0^{s} S(s-r)K(y_{r})\theta_{T_1} \omega(r)dr+\int_0^{s} AS(s-r)G(y_{r})\theta_{T_1} \omega(r)dr}, \, s\geq 0,\\
    \hat\xi(s)=u_1(s+T_1(\omega)),\quad s\in[-\mu,0],
\end{cases}
\end{equation*}
where $u_1$ is the solution to \eqref{eq?} and $s$ above is given by $s:=t-T_1.$
This means that we have to solve the same problem than in the previous step, but with initial condition $\hat\xi$ and noise $\theta_{T_1(\omega)}\omega$. Therefore, following the same steps than before, we obtain a new piece given by a local solution defined now in the interval $[T_1(\omega)-\mu, T_1(\omega)+T_1(\theta_{T_1(\omega)}\omega)]$, and thus we define $T_2(\omega)$ as

\begin{equation*}
  T_2(\omega)=T_1(\omega)+T_1(\theta_{T_1(\omega)}\omega).
\end{equation*}

Finally, to get a global mild solution to \eqref{eq?} it suffices to define appropriate stopping times, in the following way: for $i\in\NN$, considering that $T_0(\omega)=0$ and
\begin{equation}\label{eq9}
   T_i(\omega)=T_{i-1}(\omega)+T_1(\theta_{T_{i-1}(\omega)}\omega),
\end{equation}

it can be proven that then $\lim_{i\to\infty}T_i(\omega)=\infty$, see Lemma \ref{stop} below, which then concludes the proof of the existence of a global solution. Therefore we have proven the following result:

\begin{theorem}\label{t1global}
Under the assumptions of Theorem \ref{t1}, for any $\xi\in C_\mu$ and any $\omega\in\Omega$ there exists a unique mild global solution to \eqref{eq1bis}.\end{theorem}

We now prove that under our particular assumptions the cocycle property holds. Although this result is expected we demonstrate the existence of an RDS when the equation is given in the sense of mild solutions including terms stemming from the integration by parts formula.

\begin{theorem}\label{t2}
The global mild solution $u$ of (\ref{eq2}) generates a random dynamical system $$\phi:\mathbb{R}^+\times \Omega\times C_\mu \to C_\mu$$ given by $\phi(t,\omega,\xi)(\cdot)=u_t(\cdot)$, i.e.,
\begin{equation*}
\phi(t,\omega,\xi)(\cdot)=
\begin{cases}
\displaystyle{S(t+\cdot)\xi(0)+\int_0^{t+\cdot} S(t+\cdot-r)F(u_r) dr+G(u_{t+\cdot})
  \omega(t+\cdot)}\\
  + \displaystyle{\int_0^{t+\cdot}AS(t+\cdot-r)G(u_r)\omega(r)d\tau-\int_0^{t+\cdot}S(t+\cdot-r)K(u_r)\omega(r)d\tau}, \\
  \quad \text{ for } t+\cdot \geq 0,\\
\xi(t+\cdot),  \,  \text{ for } t+\cdot \leq  0.
\end{cases}\end{equation*}
Moreover,  $\xi \mapsto\phi(t,\omega,\xi)$ is continuous on $C_\mu$ for $t\ge 0$ and $\omega\in \Omega$.
\end{theorem}

\begin{proof}

In order to prove that $\phi$ is a  cocycle it is of great importance to have used (\ref{eq2}), since when we try to use directly the It\^o stochastic integral we know that exceptional sets depending on the initial condition may appear, which is in contradiction with the cocycle property.

We should have to distinguish several cases, but we present here two cases. The first one is when we consider $t,\tau\geq \mu$ so that $t+s, \tau+s\geq 0$, for all $s\in [-\mu,0]$. In that situation
\begin{align}
\label{eq4}
\begin{split}
   \phi &(t+\tau,\omega,\xi)(s)=S(t+s)\bigg(S(\tau)\xi(0)+\int_0^{\tau} S(\tau-r)F(u_r) dr\\
 & +\int_0^\tau AS(\tau-r)G(u_r)\omega(r)dr-\int_0^{\tau}S(\tau-r)K(u_r)\omega(r)dr\bigg)\\
 &+\int_0^{t+s} S(t+s-r)F(u_{\tau+r})dr+G(u_{t+\tau+s})
  \omega(t+\tau+s)\\
  & +\int_0^{t+s} AS(t+s-r)G(u_{\tau+r})\omega(\tau+r)dr\\
  &-\int_0^{t+s}S(t+s-r)K(u_{\tau+r})\omega(\tau+r)dr\\
  =&S(t+s) \phi(\tau,\omega,\xi)(0)-S(t+s)G(u_\tau)\omega(\tau)+G(u_{t+\tau+s})
  \omega(t+\tau+s)\\
  &+\int_0^{t+s} S(t+s-r)F(u_{\tau+r})dr+\int_0^{t+s} AS(t+s-r)G(u_{\tau+r})\omega(\tau+r)dr\\
  &-\int_0^{t+s}S(t+s-r)K(u_{\tau+r})\omega(r)dr.
  \end{split}
  \end{align}

Notice that
\begin{align*}
\int_0^{t+s} AS(t+s-r)G(u_{\tau+r})\omega(\tau+r)dr&=\int_0^{t+s} AS(t+s-r)G(u_{\tau+r})\theta_\tau \omega(r)dr\\
&+\int_0^{t+s} AS(t+s-r)G(u_{\tau+r})\omega(\tau)dr,\\
\int_0^{t+s}S(t+s-r)K(u_{\tau+r})\omega(\tau+r)dr&=\int_0^{t+s} S(t+s-r)K(u_{\tau+r})\theta_\tau \omega(r)dr\\
&+\int_0^{t+s} S(t+s-r)K(u_{\tau+r})\omega(\tau)dr.
\end{align*}
In addition, similar to (\ref{tr}),
\begin{align*}
&\int_0^{t+s} AS(t+s-r)G(u_{\tau+r})\omega(\tau)dr-\int_0^{t+s} S(t+s-r)K(u_{\tau+r})\omega(\tau)dr\\
=&-G(u_{t+s+\tau})\omega(\tau)+S(t+s)G(u_\tau)\omega(\tau).
\end{align*}

Hence we can rewrite \eqref{eq4} as
\begin{align*}
\begin{split}
  \phi &(t+\tau,\omega,\xi)(s)=S(t+s) \phi(\tau,\omega,\xi)(0) -S(t+s)G(u_\tau)\omega(\tau)\\
  &+\int_0^{t+s} S(t+s-r)F(u_{\tau+r})dr+\int_0^{t+s} AS(t+s-r)G(u_{\tau+r})\theta_\tau \omega(r)dr
  \\&-\int_0^{t+s} S(t+s-r)K(u_{\tau+r})\theta_\tau \omega(r)dr+G(u_{t+\tau+s})
  \omega(t+\tau+s)\\&+S(t+s)G(u_\tau)\omega(\tau)-G(u_{t+s+\tau})\omega(\tau)\\
  &=S(t+s) \phi(\tau,\omega,\xi)(0) +\int_0^{t+s} S(t+s-r)F(u_{\tau+r})dr\\
  &+\int_0^{t+s} AS(t+s-r)G(u_{\tau+r})\theta_\tau \omega(r)dr\\
      &-\int_0^{t+s} S(t+s-r)K(u_{\tau+r})\theta_\tau \omega(r)dr+G(u_{t+\tau+s})\theta_\tau \omega(t+s).
  \end{split}
\end{align*}
Defining the auxiliary function $y_p=u_{\tau+p}$ the previous expression can be rewritten as
\begin{align*}
\begin{split}
  \phi (t+\tau,\omega,\xi)(s)& =S(t+s) y(0) +\int_0^{t+s} S(t+s-r)F(y_{r})dr\\
  &+\int_0^{t+s} AS(t+s-r)G(y_{r})\theta_\tau \omega(r)dr\\
      &-\int_0^{t+s} S(t+s-r)K(y_{r})\theta_\tau \omega(r)dr+G(y_{t+s})\theta_\tau \omega(t+s)\\
      &= \phi(t,\theta_\tau \omega,\phi(\tau,\omega, \xi))(s)
  \end{split}
\end{align*}

which proves the cocycle property in that situation by the uniqueness conclusion. \\

Let us consider now the case in which $t+s+\tau \leq 0$, for $s\in [-\mu,0]$. Then it is straightforward to see that
\begin{align*}
  \phi(t+\tau,\omega,\xi)(s)&=\xi(t+\tau+s)=\phi(\tau,\omega,\xi)(t+s)= \phi(t,\theta_\tau \omega,\phi(\tau,\omega, \xi))(s).
\end{align*}
The rest of cases are left to the reader.
\end{proof}

As we have seen, for the proof of existence of a global solution we need to define a sequence of stopping times having a particular limit behavior, which is analyzed in the following result.

\begin{lemma}\label{stop}
Consider an  $\omega\in\Omega$ where $\Omega$ is defined at the beginning of this section. Suppose that the sequence $(T_i(\omega))_{i\in\NN}$ is defined by \eqref{eq9} such that similar to \eqref{cond}
\begin{equation*}
T(\omega)=\inf\{T>0:  CT+C(T^\beta+T^{1+\beta}+T^{\beta+\nu}) \|\omega\|_{\beta,0,T}\geq 1/2\}
\end{equation*}
for some positive constant $C$. Then for any $t>0$ there exists an $i\in \NN$ such that $T_i(\omega)\ge t$.
\end{lemma}

\begin{proof}
Let $t>0$ be given. For $T(\omega)=T_1(\omega)\ge t$ there is nothing to show. For the another case, let
\begin{equation*}
  k:=\|\omega\|_{\beta,0,t}.
\end{equation*}
Then the (existing and unique) solution $s^\ast$ of the equation
\begin{equation*}
  Cs+C(s^\beta+s^{1+\beta}+s^{\beta+\nu}) k= 1/2
\end{equation*}
is trivially a positive lower bound of $T(\omega)$. Furthermore,
$$\|\theta_{T(\omega)}\omega\|_{\beta,0,t-T(\omega)}\leq \|\omega\|_{\beta, 0,t}= k$$
which implies that $s^\ast\leq T(\theta_{T(\omega)}\omega)$, and therefore $T_2(\omega)\geq 2s^\ast$. Repeating this method it turns out that there exists $i\in \NN$ such that $T_i(\omega) \geq is^\ast>t$.

\end{proof}

\begin{example}\label{example}

Now we present an example of non-linear terms satisfying the previous assumptions for the existence of solution to our delayed system. For $F$ a simple example is given by a Lipschitz function $f: H\to H$ with Lipschitz constant $L_{f}$. For $u\in C_\mu$ we then set $F(u)=f(u(0))$.

For $G$ we have the following example in mind. Let $g:H \to L(U;D((-A)^\nu))$ be a Lipschitz continuous mapping with Lipschitz constant $L_{g,\nu}$. Then $g:H \to L(U;H)$ is also Lipschitz continuous with Lipschitz constant denoted by $L_{g}$. Define for $x\in C_\mu$ the smoothing operator

\begin{equation*}
  G(x)=\frac{1}{\mu}\int_{-\mu}^0g(x(q))dq.
\end{equation*}
This integral can be interpreted as a Bochner-integral, see Yoshida \cite{Yos80} , Chapter V.5. Indeed, the mapping
\begin{equation*}
  [-\mu,0]\ni\tau\mapsto g(x(\tau))\in L(U;D((-A)^\nu)))
\end{equation*}
is uniformly continuous which allows to approximate uniformly  this function by finite valued step-functions such that this integrand is strongly measurable.
In addition $\tau\mapsto \|g(x(\tau))\|_{L(U;D((-A)^\nu))}$ is integrable.
Moreover it is easy to see that $G$ is Lipschitz continuous where the Lipschitz constant can be chosen as the same Lipschitz constant than for $g$. Furthermore, if for $t\geq  0$ we assume that $x_t \in C_\mu$,

\begin{equation*}
  \frac{d}{dt}G(x_t)= \frac{1}{\mu}(g(x_t(0))-g(x_t(-\mu)))=\frac{1}{\mu}(g(x(t))-g(x(t-\mu))),
\end{equation*}

and thus can define on $C_\mu$ the operator
\begin{equation*}
K(x)=\frac{1}{\mu}(g(x(0))-g(x(-\mu)))
\end{equation*}
such that
\begin{align*}
|K(x)-K(y)|&=\frac{1}{\mu}|g(x(0)-g(y(0))+g(y(-\mu))-g(x(-\mu))|\\
&\leq \frac{1}{\mu} L_g (|x(0)-y(0)|+|y(-\mu)-x(-\mu)|)\leq \frac{2}{\mu} L_g \|x-y\|_\mu.
\end{align*}

\end{example}

\section{Existence of random attractors}

In this section we are going to study the long-time behavior of the random dynamical system defined in Theorem \ref{t2}. In particular, we show that this random dynamical system has a random attractor. To deal with random attractors we have to introduce two classes of random variables. Let for the following $(\Omega,\fF,\PP,\theta)$
be an abstract {\em ergodic} metric dynamical system. A random variable $X$ in $\RR$ is called {\em sublinear}  if
\begin{equation*}
  \lim_{t\to\pm\infty}\frac{|X(\theta_t\omega)|}{|t|}=0
\end{equation*}
on some $(\theta_t)_{t\in\RR}$--invariant set $\hat\Omega\in \fF$ of full measure. Suppose that $t\mapsto X(\theta_t\omega)$ is continuous for $\omega\in \hat\Omega$. Then for every $\omega\in\hat\Omega,\,\eps>0$ there exists a $C(\omega,\eps)$ such that
\begin{equation*}
  |X(\theta_t\omega)|\le C(\omega,\eps)+\eps|t|\quad\text{for }t\in\RR.
\end{equation*}

Furthermore, a random variable $X\ge 0$ is called {\em tempered} if $\log^+ X$ is sublinear using the convention that $\log^+0=0$.

\begin{lemma}\label{l1}
Let $(\Omega,\fF,\PP,\theta)$ be an ergodic metric dynamical system. Suppose that $X$ is a random variable such that
\begin{equation*}
  \EE\sup_{t\in [a,b]}|X(\theta_t\omega)|<\infty
\end{equation*}

for some $a<b\in\RR$. Then $X$ is sublinear.

\end{lemma}

For the proof of these  statements we refer Arnold \cite{Arn98} Proposition 4.1.3.\\

Consider again the metric dynamical system introduced in Section 2 given by the canonical Brownian motion with the Wiener shift given by (\ref{shift}). Since in particular the covariance operator $Q$ related to the Wiener measure $\PP$ is of trace class, we can establish the following result:
\begin{lemma}\label{l3}
There exists a $(\theta_t)_{t\in\RR}$--invariant set $\hat\Omega\in\fF$ of full measure such that for any $a<0$
\begin{equation*}
 X(\omega) = \sup_{a\le s\le 0}|\omega(s)|_U
\end{equation*}
is sublinear.
\end{lemma}
\begin{proof}Since $s\mapsto \omega(s)$ is continuous, the mapping $X$ is well defined and a random variable.
Since $\omega$ has trace class covariance we know from Kunita \cite{Kun91} Theorem 1.4.1 that for any $a<0$ there exists a $c>0$ such that
\begin{equation*}
  \EE\sup_{2a\le r\le 0}|\omega(r)|_U\le c<\infty.
\end{equation*}
Hence
\begin{equation*}
\sup_{a\le t\le 0}\sup_{a\leq r\leq 0}|\theta_t\omega(r)|_U = \sup_{a\le t\le 0}\sup_{a\le r\le 0}|\omega(t+r)-\omega(t)|_U\le2\sup_{2a\le r\le 0}|\omega(r)|_U
\end{equation*}
where the right hand side has a finite expectation. Hence it suffices to apply Lemma \ref{l1}. 

Indeed, for those $\omega$ for which we have sublinear growth of $\sup_{r\in [a,0]}|\theta_t\omega(r)|$, we also have
\begin{equation*}
  \sup_{2a \leq r \leq 0}|\theta_t\omega(r)|\le \sup_{a\leq r \leq 0}|\theta_t\omega(r)|+\sup_{a \leq r \leq 0}|\theta_{t+a}\omega(r)|+|\theta_t\omega(a)|
\end{equation*}
showing the sublinear growth of the left hand side, and the same property is satisfied if we replace $[-2a,0]$ by $[-n a,0]$, $n\in \NN$.
\end{proof}

In what follows we restrict once more our random dynamical system to $\hat\Omega$ defined in Lemma \ref{l3}, and we denote this restricted metric dynamical system again by $(\Omega,\fF,\PP,\theta)$. We will always assume that $\omega$ is contained in this set $\Omega$.\\

For our purpose in this section let us introduce a  tempered  set consisting of a family of sets $D=(D(\omega))_{\omega\in\Omega}$ such that $D(\omega)\subset C_\mu$ is nonempty, closed and

\begin{equation*}
  \Omega\ni\omega\mapsto \inf_{x\in D(\omega)}\|x-y\|_\mu
\end{equation*}
is a random variable for all $y\in C_\mu$, and in addition, that
\begin{equation*}
  \Omega\in\omega\mapsto \sup_{x\in D(\omega)}\|x\|_\mu
\end{equation*}

is tempered. In particular, we denote by $\dD$ the subset of all tempered sets such that the convergence relation

\begin{equation*}
  \lim_{t\to\pm\infty}\frac{\log^+ \sup_{x\in D(\theta_t\omega)}\|x\|_\mu}{|t|}=0
\end{equation*}
for {\em all} $\omega\in\Omega$.
\begin{definition}\label{d1}
A random set $A\in\dD$ such that $A(\omega)$ is compact is called a random attractor for the random dynamical system $\phi$ given in Theorem \ref{t2} if the invariance property
\begin{equation}
  \phi(t,\omega,A(\omega))=A(\theta_t\omega),\quad\text{for all }\omega\in\Omega,\;t\ge 0,
\end{equation}
and the pullback attracting property
\begin{equation*}
  \lim_{t\to\infty}{\rm dist}_{C_\mu}(\phi(t,\theta_{-t}\omega,D(\theta_{-t}\omega)),A(\omega))=0,\quad\text{for all }D\in \dD,\;\omega\in\Omega
\end{equation*}
hold true, where $dist_{C_\mu}$ denotes de Hausdorff semidistance in $C_\mu$.
\end{definition}

The following conditions ensure the existence of a random attractor, see Flandoli and  Schmalfu{\ss}\cite{FlanSch}, {Schmalfu{\ss} \cite{Sch92}, Schmalfu{\ss} \cite{Sch00}.

\begin{theorem}\label{t1-3}
Let $\phi$ be a continuous random dynamical system. Suppose that $\phi$ has a pullback $\dD$--absorbing set $C\in \dD$, that is, for any $D\in\dD$ and $\omega\in\Omega$ there exists a $t_0=t_D(\omega)$ such that
\begin{equation}
  \phi(t,\theta_{-t}\omega,D(\theta_{-t}\omega))\subset C(\omega)\quad\text{for all }\,t\geq t_0.
\end{equation}
In addition, suppose that $C(\omega)$ is compact.
Then the random dynamical system $\phi$ has a random attractor which is unique in $\dD$, given by
$$A(\omega) := \bigcap_{s\geq 0} \bigcup_{t\geq s}
\overline{\phi(t, \theta_{-t}\omega, C(\theta_{-t}\omega))}^{C_\mu}.$$
\end{theorem}

Later on we will need to apply the Gronwall lemma. The following version can be found in \cite{gronwall}, Theorem 2.1.

\begin{lemma} \label{gronwall} Let $T>0$ and $u$, $\alpha$, $f$ and $g$ be non-negative continuous functions defined on $[0,T]$ such that
$$u(t)\leq  \alpha(t) + f(t)\int_0^t g(r) u(r) dr, \quad \text{ for } t\in [0,T].$$
Then
$$u(t)\leq  \alpha(t) + f(t)\int_0^t g(r)\alpha(r) e^{\int_r^t f(\tau) g(\tau) d\tau} ds, \quad \text{ for } t\in [0,T].$$
\end{lemma}

In what follows, we consider extra assumptions to simplify the situation. Let us assume that $G$ and $K$ are bounded by $C_{G,\nu}$ and $C_K$, respectively, that is,
\begin{equation}\label{G4}
\|G(x)\|_{L(U;D((-A)^\nu))}\leq C_{G,\nu},\qquad\text{for } x\in C_\mu,
\end{equation}
\begin{equation}\label{K}
\|K(x)\|_{L(U;H)}\leq C_{K},\qquad\text{for } x\in C_\mu.
\end{equation}
By $C_G$ we will denote the bound of of $G(x)$ in $L(U;H)$.
Note that from (\ref{eq3}) we know that $F$ in particular is linearly bounded, i.e., there exist $C_F$ and $\bar C_F$ such that
\begin{align}\label{F2}
|F(x)|\leq \bar C_F +C_F \|x\|_\mu,\quad \forall x\in C_\mu.
\end{align}

Now in a first step we establish the existence of an absorbing set. Later on, we will see that this absorbing set is also an element of the family $\dD$.

\begin{lemma}\label{absorbing}

Let $\dD=\{D(\omega)\}_{\omega \in \Omega}$ be the family of tempered sets in $C_\mu$. Under the assumptions (\ref{G4}), (\ref{K}) and (\ref{F2}), if in addition
\begin{align}\label{mu}
\mu <\frac{1}{\lambda} \log\bigg(\frac{\lambda}{C_F}\bigg)
\end{align}
then the ball in $C_\mu$ given by $B(\omega):=B(0,\rho(\omega))$, with $\rho(\omega)=R(\omega)+\delta$, is a $\dD$-absorbing set, where $R$ is given in \eqref{eq100} and $\delta$ is some positive constant.
\end{lemma}

\begin{proof}
We start by obtaining an a priori estimate of our solution in $C_\mu$. We take $\tau \geq \mu$ and estimate $\|u_\tau\|_\mu=\sup_{s\in[-\mu,0]}|u(\tau+s)|$, assuming that the fiber is given for $\theta_{-t}\omega$ for $t\geq \mu$ (note that when $\tau+s<0$ then $u(\tau+s)=\xi(\tau+s)$ and therefore $|u(\tau+s)| \leq \|\xi\|_\mu$)

Thanks to the definition of $\phi$ given in Theorem \ref{t2}, considered in the fiber $\theta_{-t}\omega$, for $\tau+s\ge 0$ we get 
\begin{align}\label{in}
\begin{split}
\phi &(\tau,\theta_{-t}\omega,\xi)(s)=u(\tau+s)=S(\tau+s)\xi(0)+\int_0^{\tau+s} S(\tau+s-r)F(u_r) dr\\
&+ G(u_{\tau+s})\theta_{-t}\omega(\tau+s)-\int_0^{\tau+s} S(\tau+s-r) K(u_r) \theta_{-t}\omega(r)dr\\
&+\int_0^{\tau+s} S(\tau+s-r) A G(u_r) \theta_{-t}\omega(r) dr.
\end{split}
\end{align}
Furthermore,
\begin{align}\label{eq6}
\begin{split}
&-\int_0^{\tau+s}S(\tau+s-r)AG(u_r)\omega(-t)dr\\
=&\int_0^{\tau+s}\frac{d}{dr}S(\tau+s-r)G(u_r)\omega(-t)dr\\
= & G(u_{\tau+s})\omega(-t)-S(\tau+s)G(u_0)\omega(-t)-\int_0^{\tau+s}S(\tau+s-r)K(u_r)\omega(-t)dr
\end{split}
\end{align}
and then, taking into account (\ref{shift}), using (\ref{eq5}), (\ref{G4}), (\ref{K}), (\ref{F2}) and \eqref{eq6}, from (\ref{in}) we have
\begin{align*}
|u(\tau+s)|&\leq e^{-\lambda(\tau+s)}|\xi(0)|+ C_G |\omega(\tau+s-t)|_U+ e^{-\lambda(\tau+s)} C_G |\omega(-t)|_U\\
&+\int_0^{\tau+s} e^{-\lambda(\tau+s-r)}(\bar C_F+C_F \|u_r\|_\mu) dr\\
&+C_K \int_0^{\tau+s} e^{-\lambda(\tau+s-r)}  |\omega(r-t)|_Udr +C_{G,\nu} \int_0^{\tau+s} \frac{e^{-\lambda(\tau+s-r)}}{(\tau+s-r)^{1-\nu}} | \omega(r-t)|_U dr.
\end{align*}
Taking supremum when $s\in[-\mu,0]$ the previous inequality yields

\begin{align}\label{3.1}
\begin{split}
\|u_\tau\|_\mu&\leq e^{\lambda \mu} e^{-\lambda \tau}\|\xi\|_\mu+ C_G \|\omega_{\tau-t}\|_\mu+ e^{\lambda \mu} e^{-\lambda \tau} C_G |\omega(-t)|_U\\&+e^{\lambda \mu} \int_0^{\tau} e^{-\lambda(\tau-r)}(\bar C_F+C_F \|u_r\|_\mu) dr\\
&+C_K e^{\lambda \mu} \int_0^{\tau} e^{-\lambda(\tau-r)}  |\omega(r-t)|_Udr \\
&+C_{G,\nu}  \sup_{s \in [-\mu, 0]} \int_0^{\tau+s} \frac{e^{-\lambda(\tau+s-r)}}{(\tau+s-r)^{1-\nu}} | \omega(r-t)|_U dr,
\end{split}
\end{align}
where here $\|\omega\|_\mu$ denotes $\sup_{s\in[-\mu,0]}|\omega(s)|_U$. Let us look carefully at the last supremum term:
\begin{align*}
 \sup_{s \in [-\mu, 0]} \int_0^{\tau+s} \frac{e^{-\lambda(\tau+s-r)}}{(\tau+s-r)^{1-\nu}} | \omega(r-t)|_U dr & = \sup_{s \in [-\mu, 0]} \int_{-s}^{\tau} \frac{e^{-\lambda(\tau-r)}}{(\tau-r)^{1-\nu}} | \omega(r+s-t)|_U dr \\& \leq  \int_{0}^{\tau} \frac{e^{-\lambda(\tau-r)}}{(\tau-r)^{1-\nu}} \|\omega_{r-t}\|_\mu dr.
\end{align*}
Coming back to (\ref{3.1}) we obtain
\begin{align*}
\|u_\tau\|_\mu&\leq \alpha(\tau) + C_F e^{\lambda \mu} \int_0^{\tau} e^{-\lambda(\tau-r)}  \|u_r\|_\mu dr
\end{align*}
where by $\alpha(\cdot)$ we denote all terms which do not contain $u$, that is,
\begin{align*}
\begin{split}
\alpha(\tau) &= e^{\lambda \mu} e^{-\lambda \tau}\|\xi\|_\mu+ C_G \|\omega_{\tau-t}\|_\mu+ e^{\lambda \mu} e^{-\lambda \tau} C_G |\omega(-t)|_U+ \bar C_F e^{\lambda \mu} \int_0^{\tau} e^{-\lambda(\tau-r)} dr\\
&+C_K e^{\lambda \mu} \int_0^{\tau} e^{-\lambda(\tau-r)}  |\omega(r-t)|_Udr +C_{G,\nu}  \int_{0}^{\tau} \frac{e^{-\lambda(\tau-r)}}{(\tau-r)^{1-\nu}} \|\omega_{r-t}\|_\mu dr
 \end{split}
\end{align*}
for all $\tau\ge 0$.
Now we apply Lemma \ref{gronwall} to $e^{\lambda \tau} \|u_\tau\|_\mu$ getting
\begin{align*}
e^{\lambda \tau} \|u_\tau\|_\mu&\leq e^{\lambda \tau} \alpha(\tau) + C_F e^{\lambda \mu} \int_0^{\tau} e^{\lambda  r} \alpha(r) e^{C_Fe^{\lambda \mu} (\tau-r)}  dr.
\end{align*}
If in the previous expression we take $\tau=t$, multiplying by $e^{-\lambda t}$ we have

\begin{align}\label{apriori}
 \|u_t\|_\mu&\leq  \alpha(t) + C_F e^{\lambda \mu} \int_0^t e^{-\lambda (t-r)} e^{C_Fe^{\lambda \mu} (t-r)} \alpha(r)   dr.
\end{align}
From this a priori estimate we shall derive the absorbing set. In order to obtain this set, we need to analyze every term of the previous estimate. For the first term on the right hand side of (\ref{apriori}) we get

\begin{align*}
\begin{split}
\alpha(t) &\leq e^{\lambda \mu} e^{-\lambda t}\|\xi\|_\mu+ C_G \|\omega_{0}\|_\mu+ e^{\lambda \mu} e^{-\lambda t} C_G |\omega(-t)|_U+ \bar C_F e^{\lambda \mu} \int_0^{t} e^{-\lambda(t-r)} dr\\
&+C_K e^{\lambda \mu} \int_0^{t} e^{-\lambda(t-r)}  |\omega(r-t)|_Udr +C_{G,\nu}  \int_{0}^{t} \frac{e^{-\lambda(t-r)}}{(t-r)^{1-\nu}} \|\omega_{r-t}\|_\mu dr\\
&\leq e^{\lambda \mu} e^{-\lambda t} (\|\xi\|_\mu+  C_G \|\omega_{-t}\|_\mu)+ C_G \|\omega_{0}\|_\mu+ \frac{\bar C_F e^{\lambda \mu} }{\lambda}  \\
&+C_K e^{\lambda \mu} \int_{-\infty}^0 e^{\lambda r}  \|\omega_r\|_\mu dr +C_{G,\nu}  \int_{-\infty}^{0} \frac{e^{\lambda r}}{(-r)^{1-\nu}} \|\omega_{r}\|_\mu dr.
 \end{split}
\end{align*}
Note that these infinite integrals exist for any $\omega\in\Omega$ which follows by Remark \ref{r20} below.

Now we examine the different terms appearing under the integral in (\ref{apriori}), i.e., the expression
\begin{align}\label{new}
\begin{split}
\int_0^t  &e^{-(\lambda-C_Fe^{\lambda \mu}) (t-r)} \bigg(e^{\lambda \mu} e^{-\lambda r} (\|\xi\|_\mu+  C_G \|\omega_{-t}\|_\mu) +C_G \|\omega_{r-t}\|_\mu\\
&+ \frac{\bar C_F e^{\lambda \mu} }{\lambda} +C_K e^{\lambda \mu} \int_{-r}^0 e^{\lambda \tau}  \|\omega_\tau\|_\mu d\tau +C_{G,\nu}  \int_{-r}^{0} \frac{e^{\lambda \tau}}{(-\tau)^{1-\nu}} \|\omega_{\tau}\|_\mu d\tau \bigg) dr\\
&=:I_1+I_2+I_3+I_4+I_5.
 \end{split}
\end{align}
For the first term in (\ref{new}) we get
\begin{align}\label{t1b}
\begin{split}
I_1 =&e^{\lambda \mu}  (\|\xi\|_\mu+  C_G \|\omega_{-t}\|_\mu) e^{ -(\lambda -C_Fe^{\lambda \mu}) t}  \int_{0}^t e^{- C_Fe^{\lambda \mu} r} dr
\\&\leq  e^{\lambda \mu}  (\|\xi\|_\mu+  C_G \|\omega_{-t}\|_\mu)   \frac{e^{ -(\lambda -C_Fe^{\lambda \mu}) t}}{C_F e^{\lambda \mu}}.
 \end{split}
\end{align}
For $I_2$ and $I_3$ we easily obtain
\begin{align}\label{t2b}
I_2 \le  C_G \int_{-\infty}^0 e^{(\lambda-C_Fe^{\lambda \mu}) r} \|\omega_{r}\|_\mu dr, \qquad \qquad
I_3  \leq  \frac{\bar C_F e^{\lambda \mu} }{\lambda} \frac{1} {\lambda- C_F e^{\lambda \mu}}.
\end{align}
For the last two terms, due to (\ref{mu}) we get
\begin{align}\label{t4}
\begin{split}
& I_4+I_5 \le \int_{-\infty}^0   e^{(\lambda -C_Fe^{\lambda \mu} ) r} \bigg(C_K e^{\lambda \mu} \int_{-\infty }^0 e^{\lambda \tau}  \|\omega_\tau\|_\mu d\tau +C_{G,\nu}  \int_{-\infty }^{0} \frac{e^{\lambda \tau}}{(-\tau)^{1-\nu}} \|\omega_{\tau}\|_\mu d\tau \bigg) dr\\
\le & \frac{1}{\lambda-C_Fe^{\lambda\mu}}\bigg(C_K e^{\lambda \mu} \int_{-\infty}^0 e^{\lambda \tau}  \|\omega_\tau\|_\mu d\tau +C_{G,\nu}  \int_{-\infty }^{0} \frac{e^{\lambda \tau}}{(-\tau)^{1-\nu}} \|\omega_{\tau}\|_\mu d\tau \bigg).
\end{split}
\end{align}

Collecting everything, from (\ref{apriori}) taking into account (\ref{t2b})-(\ref{t4}) we define:
\begin{align}\label{eq100}
\begin{split}
&R(\omega)= C_G \|\omega_{0}\|_\mu+ \frac{\bar C_F e^{\lambda \mu} }{\lambda} \\
&+C_K e^{\lambda \mu} \int_{-\infty}^0 e^{\lambda r}  \|\omega_r\|_\mu dr +C_{G,\nu}  \int_{-\infty}^{0} \frac{e^{\lambda r}}{(-r)^{1-\nu}} \|\omega_{r}\|_\mu dr\\
&+C_F e^{\lambda \mu} \bigg[  C_G \int_{-\infty}^0 e^{(\lambda-C_Fe^{\lambda \mu}) r} \|\omega_{r}\|_\mu dr+\frac{\bar C_F e^{\lambda \mu} }{\lambda(\lambda- C_F e^{\lambda \mu})}\\
&+\frac{1}{\lambda-C_Fe^{\lambda\mu}}\bigg(C_K e^{\lambda \mu} \int_{-\infty}^0 e^{\lambda \tau}  \|\omega_\tau\|_\mu d\tau +C_{G,\nu}  \int_{-\infty }^{0} \frac{e^{\lambda \tau}}{(-\tau)^{1-\nu}} \|\omega_{\tau}\|_\mu d\tau \bigg)\bigg].
\end{split}
\end{align}

There are two terms which have not been considered in the above definition of $R$, the first one coming from the definition of $\alpha$ and the second one coming from the estimation (\ref{t1b}). If now we take $D\in\dD$ and replace in these two terms $\|\xi\|_\mu$ by
$\sup_{\xi\in D(\theta_{-t}\omega)}\|\xi\|_\mu$, choosing some positive $\delta$ then we can define the absorbing time $t_D(\omega)$ for $D$ by
\begin{align*}
t_D(\omega)& =\inf\bigg\{\tilde t\geq 0: \forall t\geq \tilde t, e^{\lambda \mu} e^{-\lambda t} (\sup_{\xi\in D(\theta_{-t}\omega)}\|\xi\|_\mu+  C_G \|\omega_{-t}\|_\mu)\\
\qquad &+ C_F e^{2 \lambda \mu}  (\sup_{\xi\in D(\theta_{-t}\omega)}\|\xi\|_\mu+C_G \|\omega_{-t}\|_\mu)   \frac{e^{ -(\lambda -C_Fe^{\lambda \mu}) t}}{C_F e^{\lambda \mu}}\le \frac{\delta}{2}\bigg\}+1,
\end{align*}
which follows from (\ref{mu}) and the fact that $\lim_{t\to\infty}e^{-\lambda t}\|\omega_{-t}\|_\mu=0$ for $\omega\in\Omega$.

\end{proof}

It is worth pointing out that, due to condition (\ref{mu}), the value of $\lambda$ is restricting the maximal admissible value for the time lag $\mu$. The larger $\lambda$ is, the larger $\mu$ is allowed to be.

\begin{lemma}\label{l20}
The absorbing ball given in Lemma \ref{absorbing} belongs to $\dD$.
\end{lemma}
\begin{proof}
We only need to prove that  $R(\omega)$ given in Lemma \ref{absorbing}  is tempered. Let us start to consider

\begin{equation*}\label{eq7}
 t\mapsto \int_{-\infty}^0 e^{(\lambda-C_Fe^{\lambda \mu}) \tau}  \|(\theta_{-t}\omega)_{\tau}\|_\mu d\tau=:R_1(\theta_{-t}\omega)
\end{equation*}
for $\omega\in\Omega$.

First note that, for $\tau\le 0$, we have $ \|(\theta_{-t}\omega)_\tau\|_\mu \le \|\omega_{-t+\tau}\|_\mu+\|\omega_{-t}\|_\mu$.

On the other hand, for $r>0$ let $n^*=n^*(r)$ the smallest integer such that $n^*\mu\geq r$.  Note that then we can express $\|\omega_{-r}\|_\mu \leq \|\omega_{-n^*\mu}\|_\mu+\|\omega_{(-n^*+1)\mu}\|_\mu$. Moreover
\begin{align*}\label{split}
  \|\omega_{-n^*\mu}\|_\mu & \le \sum_{i=1}^{n^*}\sup_{s\in [-\mu,0]}|\omega(-i\mu+s)-\omega((-i+1)\mu+s)|_U+\|\omega_0\|_\mu.
  \end{align*}
  Adding to every term under the previous sum
 $$\omega(-i\mu)-\omega(-i\mu)+\omega((-i+1)\mu)-\omega((-i+1)\mu)$$
 we obtain 
 \begin{align}
  \|\omega_{-n^*\mu}\|_\mu & \le 3 \sum_{i=0}^{n^*}\|\theta_{-i\mu} \omega\|_\mu +\|\omega_0\|_\mu,
  \end{align}
 because in particular
 $$|\omega(-i\mu)-\omega((-i+1)\mu)|_U =|\theta_{(-i+1)\mu}\omega(-\mu)|\leq \|\theta_{(-i+1)\mu} \omega\|_\mu.$$

We could obtain a similar estimate for $\|\omega_{(-n^*+1)\mu}\|_\mu$. Notice that the terms under the sum of (\ref{split}) grow sublinearly with respect to $i$ by Lemma \ref{l3}, and hence $\|\omega_{-r}\|_\mu$ grows subexponentially with respect to $|r|\to\infty$ for $\omega\in\Omega$.

Choose a $0<\kappa<\lambda-C_Fe^{\lambda\mu}$, then for $t\ge 0$
\begin{align*}
   e^{-2\kappa|t|} & \int_{-\infty}^0e^{(\lambda-C_Fe^{\lambda \mu}) \tau}\|(\theta_{-t}\omega)_\tau\|_\mu d\tau \\
  & \le e^{-2\kappa t}\int_{-\infty}^{0}e^{(\lambda-C_Fe^{\lambda \mu})\tau}\|\omega_{-t+\tau}\|_\mu d\tau+\frac{e^{-2\kappa t}}{\lambda-C_Fe^{\lambda \mu}}\|\omega_{-t}\|_\mu\\
  &\le e^{-{\kappa} t} \int_{-\infty}^{0}e^{{\kappa}(-t+\tau)}\|\omega_{-t+\tau}\|_\mu d\tau+\frac{e^{-2\kappa t}}{\lambda-C_Fe^{\lambda \mu}}\|\omega_{-t}\|_\mu \\
  &\le e^{-{\kappa} t} \int_{-\infty}^{-t}e^{{\kappa}\tau}\|\omega_{\tau}\|_\mu d\tau+\frac{e^{-2\kappa t}}{\lambda-C_Fe^{\lambda \mu}}\|\omega_{-t}\|_\mu
  \end{align*}
tends to zero for $t\to\infty$, which easily follows from the above growth properties. Of course, we also have this convergence for $\kappa\ge\lambda-C_Fe^{\lambda\mu}$. We can also prove this convergence for $t\to-\infty$ taking into account that the conditions on $\omega$ ensure that
$\lim_{t\to-\infty}e^{-2\kappa |t|}\sup_{\tau\in[0,-t]}\|\omega_\tau\|_\mu=0$
such that for any $\kappa>0$
\begin{equation*}
  \lim_{t\to\pm\infty}e^{-2\kappa|t|}R_1(\theta_t\omega)=0.
\end{equation*}
We obtain the same convergence for
\begin{equation*}\label{eq71}
 t\mapsto \int_{-\infty}^{0} \frac{e^{\lambda r}}{(-r)^{1-\nu}}  \|(\theta_{-t}\omega)_{\tau}\|_\mu d\tau=:R_2(\theta_{-t}\omega),\quad\text{for }\omega\in\Omega
\end{equation*}

by the fact that $\int_{-\infty}^{0}{e^{\lambda r}}{(-r)^{\nu-1}}  d\tau <\infty$. Hence the set $B$ introduced in Lemma \ref{absorbing} is in $\dD$.
\end{proof}

\begin{remark}\label{r20}
Note that in the previous proof we have obtained that
\begin{equation*}
  \lim_{t\to-\infty}e^{2 \kappa t}\|\omega_t\|_\mu=0
\end{equation*}
for $\kappa>0$, which is deduced from the polynomial growth of the Brownian motion. Moreover, from the proof of the last lemma it follows that the integrals $R_1(\omega),\,R_2(\omega)$ are finite, for any $\omega \in \Omega$.
\end{remark}

\begin{lemma}\label{comp1}
Let $u$ be a solution to \eqref{eq1bis} and suppose that $t> \mu,\, 0<\eps<\nu$, and $\alpha \leq \min\{\nu,\beta\}$.
Then
\begin{equation*}
  \sup_{s\in [-\mu,0]}|(-A)^\eps u_t(s)|\le c(t,\omega\sup_{s\in[-\mu,t]}|u(s)|),\,\|u_t\|_{C^\alpha([-\mu,0];H)}\le c(t,\omega,\sup_{s\in [-\mu,t]}|u(s)|)
\end{equation*}
where $c(t,\omega,x)$ is bounded if the nonnegative numbers $x$ are bounded for any fixed $t>\mu,\omega\in\Omega$.
\end{lemma}
\begin{proof}

We start with the first estimate. We have
\begin{align*}
  &(-A)^{\epsilon}u_{t} (s)=(-A)^{\epsilon}S(t+s)\xi(0)+\int_0^{t+s} (-A)^{\epsilon}S(t+s-r)F(u_r) dr
  \\&+(-A)^{\epsilon}G(u_{t+s})\omega(t+s)+\int_0^{t+s}(-A)^{\epsilon}AS(t+s-r)G(u_r)\omega(r)dr\\
  &-\int_0^{t+s}(-A)^{\epsilon}S(t+s-r)K(u_r)\omega(r)dr \\
  &=I_{1}+\dots+I_{5}.
  \end{align*}
  We here only consider $I_1,\,I_2,\,I_3,\,I_4$ because $I_5$ can be handled in a similar manner to $I_4$. We have
\begin{align*}
 \sup_{-\mu\leq s\leq 0}|I_{1}|&= \sup_{-\mu\leq s\leq 0}|(-A)^{\epsilon}S(t+s)\xi(0)|\\
 &\leq c_{\epsilon} |\xi(0)| \sup_{-\mu\leq s\leq 0} \frac{e^{-\lambda (t+s)}}{(t+s)^{\epsilon}}\leq c_{\epsilon}e^{-\lambda t} \|\xi\|_{\mu}\sup_{-\mu\leq s\leq 0}\frac{e^{-\lambda s}}{(t+s)^{\epsilon}}.
 \end{align*}
We have for $I_2$:
\begin{align*}
 \sup_{-\mu\leq s\leq 0}|I_{2}|&=
 \sup_{-\mu\leq s\leq 0}| \int_0^{t+s} (-A)^{\epsilon}S(t+s-r)F(u_r) dr|\\
 &\leq \sup_{-\mu\leq s\leq 0} \int_0^{t+s}\frac{e^{-\lambda (t+s-r)}}{(t+s-r)^{\epsilon}}
 (\bar C_F+C_F \|u_r\|_\mu) dr\\
&\leq  \sup_{-\mu\leq s\leq 0} \int_{-s}^{t}\frac{e^{-\lambda (t-r)}}{(t-r)^{\epsilon}}
 (\bar C_F+C_F \|u_{r+s}\|_\mu) dr\\
 &\leq\left(\bar C_F+
C_F \sup_{0\leq \tau\leq t}   \|u_{\tau}\|_\mu\right)\int_{0}^{t}\frac{e^{-\lambda (t-r)}}{(t-r)^{\epsilon}}dr.  \end{align*}

In addition, by the continuous embedding $D((-A)^\nu)\subset D((-A)^\eps)$,
 \begin{align*}
 \sup_{-\mu\leq s\leq 0}|I_{3}|&=  \sup_{-\mu\leq s\leq 0}
 |(-A)^{\epsilon}G(u_{t+s})\omega(t+s)|\\&\leq \sup_{-\mu\leq s\leq 0}
 |(-A)^{\nu}G(u_{t+s})\omega(t+s)|\leq C_{G,\nu} \|\omega_{t}\|_{\mu} .
 \end{align*}

Finally, by \eqref{eq5} we obtain
 \begin{align*}
 \sup_{-\mu\leq s\leq 0}|I_{4}|&=  \sup_{-\mu\leq s\leq 0}
| \int_0^{t+s}(-A)^{\epsilon}AS(t+s-r)G(u_r)\omega(r)dr| \\
&\leq C_{G,\nu} \sup_{-\mu\leq s\leq 0}  \int_0^{t+s} \frac{e^{-\lambda (t+s-r) }}{(t+s-r)^{1+\epsilon-\nu}}|\omega(r)|_Udr\\
&\leq  C_{G,\nu} \sup_{-\mu\leq s\leq 0}  \int_{-s}^{t} \frac{e^{-\lambda (t-r) }}{(t-r)^{1+\epsilon-\nu}}
 |\omega(r+s)|_Udr \leq  C_{G,\nu}
\int_{0}^{t} \frac{e^{-\lambda (t-r) }}{(t-r)^{1+\epsilon-\nu}} \|\omega_{r}\|_{\mu}dr.
\end{align*}

Therefore, we have already proven the first statement of this result. In order to prove the second one, consider $s_{1} \leq s_{2}\in [-\mu,0]$. We have that
\begin{align}\label{uholder}
\begin{split}
& u_{t} (s_{1})- u_{t} (s_{2})=(S(t+s_{1})-S(t+s_{2}))\xi(0)\\
&+G(u_{t+s_{1}})\omega(t+s_{1})-
G(u_{t+s_{2}})\omega(t+s_{2})\\
&+\int_0^{t+s_{1}} S(t+s_{1}-r)F(u_r) dr-\int_0^{t+s_{2}} S(t+s_{2}-r)F(u_r) dr \\
&+\int_0^{t+s_{1}}AS(t+s_{1}-r)G(u_r)\omega(r)dr-\int_0^{t+s_{2}}AS(t+s_{2}-r)G(u_r)\omega(r)dr \\
&- \int_0^{t+s_{1}}S(t+s_{1}-r)K(u_r)\omega(r)dr +\int_0^{t+s_{2}}S(t+s_{2}-r)K(u_r)\omega(r)dr.
\end{split}
 \end{align}

On account of \eqref{eq30},
\begin{align*}
|(S(t+s_{1})-S(t+s_{2}))\xi(0)|&\leq
 |s_{1}-s_{2}|^{\alpha}|t+s_{1}|^{-\alpha}|\xi(0)| \leq |s_{1}-s_{2}|^{\alpha}|t-\mu|^{-\alpha}\|\xi\|_{\mu}.
\end{align*}

In addition
\begin{align*}&\bigg|\int_0^{t+s_{1}}AS(t+s_{1}-r)G(u_r)\omega(r)dr-\int_0^{t+s_{2}}AS(t+s_{2}-r)G(u_r)\omega(r)dr\bigg|\\
&\leq  \bigg|\int_{-s_1}^{t}A(S(t-r)-S(t+s_{2}-s_1-r))G(u_{r+s_1})\omega(r+s_1)dr\bigg| \\
&+
\bigg|\int_{t+s_{1}}^{t+s_{2}}  AS(t+s_{2}-r)G(u_{r})\omega(r)dr\bigg|\\
&\leq C_{G,\nu} |s_{1}-s_{2}|^{\alpha}\int_0^{t+s_{1}}\frac{1}{|t+s_{1}-r|^{1+\alpha-\nu}}|\omega(r)|_Udr\\
&+C_{G,\nu}\int_{t+s_{1}}^{t+s_{2}} \frac{1}{|t+s_{2}-r|^{1-\nu}}|\omega(r)|_Udr\\
&\leq  C_{G,\nu}|s_{1}-s_{2}|^{\alpha}
\left(\int_0^{t}\frac{dr}{|t-r|^{1+\alpha-\nu}}+
|s_{1}-s_{2}|^{\nu-\alpha}\right)\|\omega_{t}\|_\mu\le C (G,\nu,\mu)|s_{1}-s_{2}|^{\alpha}.
\end{align*}

The terms stemming from $F$ and $K$ in (\ref{uholder}) can be estimated in a similar manner, taking the linear grow condition of $F$ and the boundedness of $K$  into account.
Now, let us focus on the second term in (\ref{uholder}), which can be expressed in the following way

\begin{align*}
&|G(u_{t+s_{1}})\omega(t+s_{1})-G(u_{t+s_{2}})\omega(t+s_{2})|\\
&\leq |(G(u_{t+s_{1}})-G(u_{t+s_{2}})) \omega(t+s_{1})|
+|G(u_{t+s_{2}})(\omega(t+s_{1})-\omega(t+s_{2}))|=:J_{1}+J_2.
\end{align*}
It is easy to see that
\begin{align*}
J_{2}\leq C_{G}\|\omega\|_{ C^{\beta}([-\mu,t]; U)}
|s_{1}-s_{2}|^{\beta}.
\end{align*}

Furthermore, since the mapping
\begin{equation*}
  t\mapsto G(u_t)\in C^1([0,T];L(U;H)),
\end{equation*}

then,  we can deduce that there exists $\tau\in [t+s_{1},t+s_{2}]$ such that
\begin{align*}
J_{1}&\leq |K(u_{\tau}) \omega(t+s_{1})||s_{1}-s_{2}|\leq  C_{K}|s_{1}-s_{2}| \|\omega_{t}\|_{\mu}.
 \end{align*}

Hence we can estimate $\|u_{t}\|_{C^{\alpha}([-\mu,0]; H)}$.
\end{proof}

\begin{lemma}\label{comp}
Let $B\in \dD$ be the absorbing set from Lemma \ref{absorbing}. Then there exists a compact absorbing set $C\in \dD$.
\end{lemma}
\begin{proof}
Since $B\in \dD$, from Lemma \ref{absorbing} we know that $B$ absorbs itself. Let $t_B(\omega)$ be the absorbing time for $B$, that is, $\phi(t,\theta_{-t}\omega,B(\theta_{-t}\omega))\subset B(\omega)$ for $t\ge t_B(\omega)$.
Indeed we choose a $t>0$ such that
\begin{equation*}
  t\ge t_B(\omega)+2\mu,\quad t\ge t_B(\theta_{-\mu}\omega)+\mu,\quad t\ge t_{B}(\theta_{-2\mu}\omega)+2\mu.
\end{equation*}
Then, due to the absorbing property,
\begin{align*}
  &\phi(t,\theta_{-t}\omega,B(\theta_{-t}\omega))\subset B(\omega)\\
  &\phi(t-\mu,\theta_{-t}\omega,B(\theta_{-t}\omega))\subset B(\theta_{-\mu}\omega)\\
  &\phi(t-2\mu,\theta_{-t}\omega,B(\theta_{-t}\omega))\subset B(\theta_{-2\mu}\omega).
\end{align*}
Note that by the last inclusion we also get
$$\phi(t,\theta_{-t}\omega,B(\theta_{-t}\omega))\subset \phi(2\mu,\theta_{-2\mu}\omega,B(\theta_{-2\mu}\omega)).$$
Let $u$ be a solution to \eqref{eq1bis} with initial function $\xi\in B(\theta_{-2\mu}\omega)$ and with noise path $\theta_{-2\mu}\omega$.

We can also rewrite the previous expressions as
\begin{equation*}
     \sup_{r\in[\mu,2\mu]}|u(r)|\le \rho(\omega),\,  \sup_{r\in[0,\mu]}|u(r)|\le \rho(\theta_{-\mu}\omega),\, \sup_{r\in[-\mu,0]}|u(r)|\le \rho(\theta_{-2\mu}\omega).
\end{equation*}
As a consequence, on the one hand, according to Lemma \ref{comp1}, if now we consider $c(2\mu,\theta_{-2\mu}\omega,\max_{i\in\{0,1,2\}}\rho(\theta_{-i\mu}\omega))$, we can apply Arzela-Ascoli's theorem to obtain the compactness of $\overline{\phi(2\mu,\theta_{-2\mu}\omega,B(\theta_{-2\mu}\omega))}^{C_\mu}$. Therefore we can consider 
$$C(\omega)=\overline{\phi(t,\theta_{-t}\omega,B(\theta_{-t}\omega))}^{C_\mu}$$ which is also compact. On the other hand, $C(\omega)\subset B(\omega)$ such that $C\in\dD$. 

Note that we can apply Arzela-Ascoli's theorem, because the H\"older continuity estimate obtained in Lemma \ref{comp1} implies the equicontinuity, and the boundeness of $u_t(s)$ in the norm of $D((-A)^\epsilon)$ implies the boundedness of the norm of $u_t(s)$ in $H$, since $D((-A)^\epsilon)$ is compactly embedded into $H$, see \cite{MA}, Chapter 7.

Furthermore $C$ is pullback absorbing because of
\begin{equation*}
  \phi(s,\theta_{-s-t}\omega,D(\theta_{-s-t}\omega))\subset B(\theta_{-t}\omega)
\end{equation*}
for $s\ge t_D(\theta_{-t}\omega)$ and $D\in\dD$.
\end{proof}
Finally, as a direct application of Theorem \ref{t1-3}, we have
\begin{theorem}
The random dynamical system generated by \eqref{eq1bis} has a random attractor.
\end{theorem}

\end{document}